\documentclass[a4paper,11pt,fleqn]{article}
\usepackage{amsmath}
\usepackage{amssymb}
\usepackage{theorem}
\usepackage{mathrsfs} 
\usepackage[usenames,dvipsnames]{pstricks}
\usepackage{euscript}
\usepackage{graphicx}
\usepackage{enumitem}
\usepackage{charter}
\newcommand{\email}[1]{\href{mailto:#1}{\nolinkurl{#1}}}
\definecolor{labelkey}{rgb}{0,0.08,0.45}
\definecolor{refkey}{rgb}{0,0.6,0.0}
\definecolor{Brown}{rgb}{0.45,0.0,0.05}
\definecolor{dgreen}{rgb}{0.00,0.49,0.00}
\definecolor{dblue}{rgb}{0,0.08,0.75}
\RequirePackage[colorlinks,hyperindex]{hyperref} 
\hypersetup{linktocpage=true,citecolor=dblue,linkcolor=dgreen}
\PassOptionsToPackage{normalem}{ulem}
\usepackage{ulem}

\renewcommand{\leq}{\ensuremath{\leqslant}}
\renewcommand{\geq}{\ensuremath{\geqslant}}
\renewcommand{\le}{\ensuremath{\leqslant}}

\newcommand{\minimize}[2]{\ensuremath{\underset{\substack{{#1}}}%
{\text{\rm minimize}}\;\;#2}}
\newcommand{\EC}[2]{{\mathsf E}(#1\!\mid\! #2)} 
\newcommand{\EEC}[2]{{\mathsf E}\bigg(#1\!\:\Big|\: #2\bigg)}

\newcommand{\scal}[2]{{\left\langle{{#1}\mid{#2}}\right\rangle}}

\newcommand{\menge}[2]{\big\{{#1}~\big |~{#2}\big\}} 
\newcommand{\Menge}[2]{\left\{{#1}~\Big|~{#2}\right\}} 
\newcommand{\HHH}{{\ensuremath{\boldsymbol{\mathsf H}}}}

\newcommand{\HH}{\ensuremath{{\mathsf H}}}

\newcommand{\FF}{\ensuremath{{\EuScript F}}}

\newcommand{\emp}{\ensuremath{{\varnothing}}}

\newcommand{\Id}{{\sf{Id}}\,}
\newcommand{\ID}{{\boldsymbol{\sf{Id}}\,}}

\newcommand{\Cart}{\ensuremath{\raisebox{-0.5mm}{\mbox{\huge{$\times$}}}}}
\newcommand{\RR}{\ensuremath{\mathbb{R}}}
\newcommand{\RP}{\ensuremath{\left[0,+\infty\right[}}

\newcommand{\RPP}{\ensuremath{\left]0,+\infty\right[}}

\newcommand{\RX}{\ensuremath{\left]-\infty,+\infty\right]}}

\newcommand{\EE}{\ensuremath{\mathsf E}}
\newcommand{\PP}{\ensuremath{\mathsf P}}
\newcommand{\as}{\ensuremath{\text{\rm $\PP$-a.s.}}}

\newcommand{\NN}{\ensuremath{\mathbb N}}

\newcommand{\exi}{\ensuremath{\exists\,}}

\newcommand{\zer}{\ensuremath{\text{\rm zer}\,}}
\newcommand{\pinf}{\ensuremath{{+\infty}}}

\newcommand{\prox}{\ensuremath{\text{\rm prox}}}
\newcommand{\Fix}{\ensuremath{\text{\rm Fix}\,}}

\newcommand{\zeroun}{\ensuremath{\left]0,1\right[}}   
\newcommand{\rzeroun}{\ensuremath{\left]0,1\right]}}   
\newcommand{\lzeroun}{\ensuremath{\left[0,1\right[}}   

\newtheorem{theorem}{Theorem}[section]
\newtheorem{lemma}[theorem]{Lemma}
\newtheorem{corollary}[theorem]{Corollary}

\theoremstyle{plain}{\theorembodyfont{\rmfamily}%
}
\theoremstyle{plain}{\theorembodyfont{\rmfamily}%
}
\theoremstyle{plain}{\theorembodyfont{\rmfamily}%
\newtheorem{algorithm}[theorem]{Algorithm}}
\theoremstyle{plain}{\theorembodyfont{\rmfamily}%
\newtheorem{problem}[theorem]{Problem}}
\theoremstyle{plain}{\theorembodyfont{\rmfamily}%
\newtheorem{example}[theorem]{Example}}
\theoremstyle{plain}{\theorembodyfont{\rmfamily}%
\newtheorem{remark}[theorem]{Remark}}
\theoremstyle{plain}{\theorembodyfont{\rmfamily}%
}
\theoremstyle{plain}{\theorembodyfont{\rmfamily}%
}

\numberwithin{equation}{section}
\begin{document}
\title{\sffamily Stochastic Quasi-Fej\'er Block-Coordinate Fixed 
Point Iterations With Random Sweeping II: Mean-Square and Linear
Convergence\thanks{Contact author: 
P. L. Combettes, \email{plc@math.ncsu.edu}, phone:+1 (919)
515-2671. The work of P. L. Combettes was partially supported by
the National Science Foundation under grant CCF-1715671.}}
\author{Patrick L. Combettes$^{1}$ and 
Jean-Christophe Pesquet$^2$
\\[5mm]
\small
\small $\!^1$North Carolina State University\\
\small Department of Mathematics\\
\small Raleigh, NC 27695-8205, USA\\
\small\email{plc@math.ncsu.edu}\\
\\[2mm]
\small
\small $\!^2$CentraleSup\'elec, Universit\'e Paris-Saclay\\
\small Center for Visual Computing\\
\small 92295 Ch\^atenay-Malabry, France\\
\small \email{jean-christophe@pesquet.eu}
}
\date{~}

\thispagestyle{empty}

\maketitle

\vskip 8mm

\begin{abstract}
Reference \cite{Siop15} investigated the almost sure weak
convergence of block-coordinate fixed point algorithms and
discussed their applications to nonlinear analysis and optimization.
This algorithmic framework features random sweeping rules to select
arbitrarily
the blocks of variables that are activated over the course of the
iterations and it allows for stochastic errors in the evaluation of
the operators. The present paper establishes results on the
mean-square and linear convergence of the iterates. Applications to
monotone operator splitting and proximal optimization algorithms
are presented.  
\end{abstract} 

{\bfseries Keywords.} Block-coordinate algorithm,
fixed-point algorithm,
mean-square convergence,
monotone operator splitting,
linear convergence,
stochastic algorithm
\newpage

\section{Introduction}

In \cite{Siop15}, we investigated the asymptotic behavior 
of abstract stochastic quasi-Fej\'er fixed point iterations in
a Hilbert space $\HHH$ and applied these results
to establish almost sure convergence properties for randomly 
activated block-coordinate, stochastically perturbed extensions of 
algorithms employed in fixed point theory, monotone operator 
splitting, and optimization. The basic property of the operators 
used in the underlying model was that of quasinonexpansiveness.
Recall that an operator $\boldsymbol{\mathsf{T}}\colon\HHH\to\HHH$
with fixed point set $\Fix\boldsymbol{\mathsf{T}}$ is 
quasinonexpansive if 
\begin{equation}
\label{e:quasi}
(\forall\boldsymbol{\mathsf{z}}\in\Fix
\boldsymbol{\mathsf{T}})(\forall\boldsymbol{\mathsf{x}}\in\HHH)
\quad\|\boldsymbol{\mathsf{T}}\boldsymbol{\mathsf{x}}-
\boldsymbol{\mathsf{z}}\|\leq\|\boldsymbol{\mathsf{x}}-
\boldsymbol{\mathsf{z}}\|,
\end{equation}
and strictly quasinonexpansive if the above inequality is strict
whenever $\boldsymbol{\mathsf{x}}\not\in\Fix\boldsymbol{\mathsf{T}}$ 
\cite{Livre1}. The fixed point problem under investigation in
\cite{Siop15} was the following. 

\begin{problem}
\label{prob:0}
Let $(\HH_i)_{1\leq i\leq m}$ be separable real Hilbert spaces and 
let $\HHH=\HH_1\oplus\cdots\oplus\HH_m$ be their direct Hilbert sum.
For every $n\in\NN$, let 
$\boldsymbol{\mathsf T}_{\!n}\colon\HHH\to\HHH\colon
\boldsymbol{\mathsf{x}}\mapsto({\mathsf T}_{\!i,n}\,
\boldsymbol{\mathsf{x}})_{1\leq i\leq m}$ be a quasinonexpansive 
operator where, for every $i\in\{1,\ldots,m\}$, 
${\mathsf T}_{\!i,n}\colon\HHH\to\HH_i$ is measurable. Suppose
that $\boldsymbol{\mathsf F}=\bigcap_{n\in\NN}
\Fix\boldsymbol{\mathsf T}_{\!n}\neq\emp$. The problem is to find
a point in $\boldsymbol{\mathsf F}$.
\end{problem}

In \cite{Siop15}, Problem~\ref{prob:0} was solved via the
following block-coordinate algorithm. The main advantages of
a block-coordinate strategy is to reduce the computational load
and the memory requirements per iteration. In addition, our
approach adopts random sweeping rules to select arbitrarily the
blocks of variables that are activated at each iteration, and 
it allows for stochastic errors in the implementation of the 
operators.

\begin{algorithm} 
\label{algo:1}
Let $(\lambda_n)_{n\in\NN}$ be a sequence in $\rzeroun$ and set
$\mathsf{D}=\{0,1\}^m\smallsetminus\{\boldsymbol{\mathsf{0}}\}$. 
Let $\boldsymbol{x}_0$ and $(\boldsymbol{a}_n)_{n\in\NN}$ 
be $\HHH$-valued random variables, and let 
$(\boldsymbol{\varepsilon}_n)_{n\in\NN}$ be identically distributed 
$\mathsf{D}$-valued random variables. Iterate
\begin{equation}
\label{e:2014}
\begin{array}{l}
\text{for}\;n=0,1,\ldots\\
\left\lfloor
\begin{array}{l}
\text{for}\;i=1,\ldots,m\\
\left\lfloor
\begin{array}{l}
x_{i,n+1}=x_{i,n}+\varepsilon_{i,n}\lambda_n\big(
{\mathsf T}_{\!i,n}\,(x_{1,n},\ldots,x_{m,n})+a_{i,n}-x_{i,n}\big).
\end{array} 
\right.
\end{array} 
\right.
\end{array} 
\end{equation}
\end{algorithm}

At iteration $n$ of Algorithm~\ref{algo:1}, 
$\lambda_n\in\rzeroun$ is a relaxation
parameter, $a_{i,n}$ an $\HH_i$-valued random variable modeling some
stochastic error in the application of the operator 
$\mathsf{T}_{\!i,n}$,
and $\varepsilon_{i,n}$ an $\{0,1\}$-valued random variable that
signals the activation of the $i$th block 
${\mathsf T}_{\!i,n}$ of the operator 
$\boldsymbol{\mathsf T}_{\!n}$. Almost sure  
weak and strong convergence properties of this scheme were
established in \cite{Siop15}. In the present paper, we
complement these results by proving mean-square and linear
convergence properties for the orbits of \eqref{e:2014} 
under the additional
assumption that each operator $\boldsymbol{\mathsf T}_{\!n}$ in 
Problem~\ref{prob:0} satisfies the property
\begin{equation}
\label{e:squasi}
(\exi\tau_n\in\lzeroun)(\forall\boldsymbol{\mathsf{z}}\in\Fix
\boldsymbol{\mathsf{T}}_{\!n})(\forall\boldsymbol{\mathsf{x}}
\in\HHH)
\quad\|\boldsymbol{\mathsf{T}}_{\!n}\boldsymbol{\mathsf{x}}-
\boldsymbol{\mathsf{z}}\|\leq\sqrt{\tau_n}
\|\boldsymbol{\mathsf{x}}-\boldsymbol{\mathsf{z}}\|,
\end{equation}
which implies that $\boldsymbol{\mathsf{T}}_{\!n}$ is 
strictly quasinonexpansive and that
$\Fix\boldsymbol{\mathsf{T}}_{\!n}$ is a
singleton. Our results appear to be the first of this kind
regarding the block-coordinate algorithm \eqref{e:2014}, even
in the case of a single-block, when it reduces to the
stochastically perturbed iteration
\begin{equation}
\label{e:2014-}
\begin{array}{l}
\text{for}\;n=0,1,\ldots\\
\left\lfloor
\begin{array}{l}
x_{n+1}=x_n+\lambda_n\big({\mathsf T}_{\!n}x_n+a_{n}-x_n\big),
\end{array} 
\right.
\end{array} 
\end{equation}
special cases of which are studied in \cite{Atch17,Pafa16,Rosa16}.

The problem we address is more precisely described as
follows.

\begin{problem}
\label{prob:1}
Let $(\HH_i)_{1\leq i\leq m}$ be separable real Hilbert spaces,
set $\HHH=\HH_1\oplus\cdots\oplus\HH_m$, and let 
$\{\tau_{i,n}\}_{1\leq i\leq m}\subset\lzeroun$.
For every $n\in\NN$, let 
$\boldsymbol{\mathsf T}_{\!n}\colon\HHH\to\HHH\colon
\boldsymbol{\mathsf x}\mapsto({\mathsf T}_{\!i,n}\,
\boldsymbol{\mathsf x})_{1\leq i\leq m}$ be measurable and
quasinonexpansive with common fixed point 
$\overline{\boldsymbol{\mathsf{x}}}=
(\overline{\mathsf{x}}_i)_{1\leq i\leq m}$, and such that
\begin{equation}
\label{e:1}
(\forall n\in\NN)(\forall\boldsymbol{\mathsf{x}}\in\HHH)
\quad\|\boldsymbol{\mathsf{T}}_{\!n}\boldsymbol{\mathsf{x}}-
\overline{\boldsymbol{\mathsf{x}}}\|^2\leq\sum_{i=1}^m
\tau_{i,n}\|\mathsf{x}_i-\overline{\mathsf{x}}_i\|^2.
\end{equation}
The problem is to find $\overline{\boldsymbol{\mathsf{x}}}$.
\end{problem}

The proposed mean-square convergence
results are the most comprehensive available to date for stochastic
block-iterative fixed point methods at the level of generality
and flexibility of Algorithm \eqref{e:2014}. Special cases
concerning finite-dimensional minimization problems involving a
smooth function with restrictions in the implementation of 
\eqref{e:2014} are discussed in \cite{Nest12,Rich14,Rich16}.

The remainder of the paper consists of 3 sections. In
Section~\ref{sec:2}, we provide our notation and preliminary
results. Section~\ref{sec:3} is dedicated to the mean-square
convergence analysis of Algorithm~\ref{algo:1} and it discusses 
its linear convergence properties. Applications are presented in 
Section~\ref{sec:4}.

\section{Notation, background, and preliminary results}
\label{sec:2}

{\bfseries Notation.}
$\HH$ is a separable real Hilbert space with 
scalar product $\scal{\cdot}{\cdot}$, associated norm 
$\|\cdot\|$, Borel $\sigma$-algebra $\mathcal{B}$, and identity
operator $\Id$. The underlying probability space is 
$(\Omega,\FF,\PP)$. A $\HH$-valued random variable is a measurable 
map $x\colon(\Omega,\FF)\to(\HH,\mathcal{B})$ \cite{Fort95,Ledo91}.
The $\sigma$-algebra generated by a family $\Phi$ of 
random variables is denoted by $\sigma(\Phi)$.
Let $\mathscr{F}=(\FF_n)_{n\in\NN}$ be a sequence of 
sub-sigma algebras of $\FF$ such that $(\forall n\in\NN)$ 
$\FF_n\subset\FF_{n+1}$.
We denote by $\ell_+(\mathscr{F})$ the set of sequences 
of $\RP$-valued random variables $(\xi_n)_{n\in\NN}$ such that,
for every $n\in\NN$, $\xi_n$ is $\FF_n$-measurable. We set
\begin{equation}
\label{e:2013-11-13}
(\forall p\in\RPP)\quad\ell_+^p(\mathscr{F})=
\Menge{(\xi_n)_{n\in\NN}\in\ell_+(\mathscr{F})}
{\sum_{n\in\NN}\xi_n^p<\pinf\quad\as}.
\end{equation}

\begin{lemma}
\label{l:2016}
Let $\mathscr{F}=(\FF_n)_{n\in\NN}$ be a sequence of sub-sigma 
algebras of $\FF$ such that 
$(\forall n\in\NN)$ $\FF_n\subset\FF_{n+1}$. Let
$(\alpha_n)_{n\in\NN}\in\ell_+({\mathscr{F}})$, 
let $(\vartheta_n)_{n\in\NN}\in\ell_+({\mathscr{F}})$, 
let $(\eta_n)_{n\in\NN}\in\ell_+({\mathscr{F}})$, and 
suppose that there exists a sequence $(\chi_n)_{n\in\NN}$ in
$\RP$ such that $\varlimsup\chi_n<1$ and
\begin{equation}
\label{e;2016-04-24}
(\forall n\in\NN)\quad
\EC{\alpha_{n+1}}{\FF_n}+\vartheta_n\leq\chi_n\alpha_n
+\eta_n\quad\as
\end{equation}
Then the following hold:
\begin{enumerate}
\item  
\label{le:1i} 
Set $(\forall n\in\NN)$
$\overline{\vartheta}_n=\sum_{k=0}^n\big(\prod_{\ell=k+1}^n
\chi_\ell\big)\EC{\vartheta_k}{\FF_0}$ and
$\overline{\eta}_n=\sum_{k=0}^n\big(\prod_{\ell=k+1}^n
\chi_\ell\big)\EC{\eta_k}{\FF_0}$
(with the convention $\prod_{n+1}^n\cdot=1$). Then
\begin{equation}
\label{e:newequa}
(\forall n\in\NN)\quad
\EC{\alpha_{n+1}}{\FF_0}+
\overline{\vartheta}_n\leq\Bigg(\prod_{k=0}^n\chi_k\Bigg)
\alpha_0+\overline{\eta}_n\quad\as
\end{equation}
\item 
\label{le:1ii}  
Suppose that $\EE\alpha_0<\pinf$  and 
$\sum_{n\in\NN}\EE\eta_n<\pinf$.
Then $\sum_{n\in\NN}\EE\alpha_n<\pinf$ and
$\sum_{n\in\NN}\EE\vartheta_n<\pinf$.
\end{enumerate}
\end{lemma}
\begin{proof}
\ref{le:1i}:
Let $n\in\NN\smallsetminus\{0\}$. We deduce from 
\eqref{e;2016-04-24} that
\begin{align}
\EC{\EC{\alpha_{n+1}}{\FF_n}}{\FF_{n-1}}+\EC{\vartheta_n}{\FF_{n-1}}
&\leq\EC{\chi_n\alpha_n}{\FF_{n-1}}+\EC{\eta_n}{\FF_{n-1}}
\nonumber\\
&=\chi_n\EC{\alpha_n}{\FF_{n-1}}+\EC{\eta_n}{\FF_{n-1}}\quad\as
\label{e:2016-04-24bis}
\end{align}
However, since $\FF_{n-1}\subset\FF_n$, we have 
$\EC{\EC{\alpha_{n+1}}{\FF_n}}{\FF_{n-1}}=
\EC{\alpha_{n+1}}{\FF_{n-1}}$. Therefore \eqref{e:2016-04-24bis}
yields
\begin{equation}
\EC{\alpha_{n+1}}{\FF_{n-1}}\leq\chi_n\EC{\alpha_n}{\FF_{n-1}}
+\EC{\eta_n}{\FF_{n-1}}-\EC{\vartheta_n}{\FF_{n-1}}\quad\as
\end{equation}
By proceeding by induction and observing that $\alpha_0$ is
$\FF_0$-measurable, we obtain \eqref{e:newequa}.

\ref{le:1ii}:
We derive from \eqref{e:newequa} that
\begin{equation}
\label{e:newequabis}
(\forall n\in\NN)\quad
\EE\alpha_{n+1}+\EE\overline{\vartheta}_n\leq
\Bigg(\prod_{k=0}^n\chi_k\Bigg)\EE\alpha_0+\EE\overline{\eta}_n=
\Bigg(\prod_{k=0}^n\chi_k\Bigg)
\EE\alpha_0+\sum_{k=0}^n\Bigg(\prod_{\ell=k+1}^n
\chi_\ell\Bigg)\EE\eta_k.
\end{equation}
On the other hand, there exist 
$q\in\NN$ and ${\rho}\in\zeroun$ such that, for every integer 
$n>q$, $\chi_n<\rho$ and, therefore, 
\begin{align}
\label{e:newequabis2}
\EE\alpha_{n+1}+\EE\overline{\vartheta}_n&\leq
\Bigg(\prod_{k=0}^q\chi_k\Bigg){\rho}^{n-q}
\EE\alpha_0+\sum_{k=0}^q\Bigg(\prod_{\ell=k+1}^q
\chi_\ell\Bigg){\rho}^{n-q}\EE\eta_k
+\sum_{k=q+1}^n {\rho}^{n-k}\EE\eta_k\nonumber\\
&\leq\Bigg(\prod_{k=0}^q\chi_k\Bigg){\rho}^{n-q}\EE\alpha_0+
\max\left\{\left(\frac{\prod_{\ell=k+1}^q\chi_\ell}
{{\rho}^{q-k}}\right)_{0\leq k\leq q},1\right\}
\sum_{k=0}^n{\rho}^{n-k}\EE\eta_k.
\end{align}
Since $\sum_{n\in\NN}{\rho}^n<\pinf$ and
$\sum_{n\in\NN}\EE\eta_n<\pinf$, it follows from standard
properties of the discrete convolution that
$(\sum_{k=0}^n{\rho}^{n-k}\EE\eta_k)_{n\in\NN}$ is
summable. We then deduce from \eqref{e:newequabis2} that
$\sum_{n\in\NN}\EE\alpha_n<\pinf$ and
$\sum_{n\in\NN}\EE\overline{\vartheta}_n<\pinf$.
Thus, the inequalities
\begin{equation}
(\forall n\in\NN)\quad
\EE\vartheta_n\leq\sum_{k=0}^n\Bigg(\prod_{\ell=k+1}^n
\chi_\ell\Bigg)\EE\vartheta_k=\EE\overline{\vartheta}_n
\end{equation}
yield $\sum_{n\in\NN}\EE\vartheta_n<\pinf$.
\end{proof}

\begin{lemma}
\label{l:1}
Let $\phi\colon\RP\to\RP$ be a strictly increasing 
function such that $\lim_{t\to\pinf}\phi(t)=\pinf$, let 
$(x_n)_{n\in\NN}$ be a sequence of $\HH$-valued random variables,
and let $(\FF_n)_{n\in\NN}$ be a sequence
of sub-sigma-algebras of $\FF$ such that
\begin{equation}
\label{e:2013-11-14'}
(\forall n\in\NN)\quad\sigma(x_0,\ldots,x_n)\subset
\FF_n\subset\FF_{n+1}.
\end{equation}
Suppose that there exist $\mathsf{z}\in\HH$,
$(\vartheta_n)_{n\in\NN}\in\ell_+({\mathscr{F}})$,
$(\eta_n)_{n\in\NN}\in\ell_+({\mathscr{F}})$,
and a sequence $(\chi_n)_{n\in\NN}$ in $\RP$ such that
$\varlimsup\chi_n<1$ and
\begin{equation}
\label{e:sqf1}
(\forall n\in\NN)\quad
\EC{\phi(\|x_{n+1}-\mathsf{z}\|)}{\FF_n}+\vartheta_n
\leq\chi_n\phi(\|x_n-\mathsf{z}\|)+
\eta_n\quad\as
\end{equation}
Set $(\forall n\in\NN)$
$\overline{\vartheta}_n=\sum_{k=0}^n\big(\prod_{\ell=k+1}^n
\chi_\ell\big)\EC{\vartheta_k}{\FF_0}$ and
$\overline{\eta}_n=\sum_{k=0}^n\big(\prod_{\ell=k+1}^n
\chi_\ell\big)\EC{\eta_k}{\FF_0}$.
Then the following hold:
\begin{enumerate}
\item 
\label{l:1i} 
$(\forall n\in\NN)$
$\displaystyle\EC{\phi(\|x_{n+1}-\mathsf{z}\|)}
{\FF_0}+\overline{\vartheta}_n
\leq\Bigg(\prod_{k=0}^n\chi_k\Bigg)\phi(\|x_0-\mathsf{z}\|)+
\overline{\eta}_n\quad\as$
\item 
\label{l:1ii}  
Let $p\in\RPP$ and set $\phi=|\cdot|^p$. Suppose that 
$x_0\in L^p(\Omega,\FF,\PP;\HH)$ and that 
$\sum_{n\in\NN}\EE\eta_n<\pinf$. Then the following hold:
\begin{enumerate}
\item 
\label{l:1iia} 
$\EE\|x_n-\mathsf{z}\|^p\to 0$
and $\sum_{n\in\NN}\EE\vartheta_n<\pinf$.
\item 
\label{l:1iib} 
Suppose that $(\eta_n)_{n\in\NN}\in\ell_+^1(\mathscr{F})$. Then
$(x_n)_{n\in\NN}$ converges strongly $\as$ to $\mathsf{z}$.
\end{enumerate}
\end{enumerate}
\end{lemma}
\begin{proof}
We apply Lemma~\ref{l:2016}\ref{le:1i} with
$(\forall n\in\NN)$ $\alpha_n=\phi(\|x_{n}-\mathsf{z}\|)$.

\ref{l:1i}: See Lemma~\ref{l:2016}\ref{le:1i}.

\ref{l:1iia}: 
Since $L^p(\Omega,\FF,\PP;\HH)$ is a vector space 
\cite[Th\'eor\`eme~5.8.8 and Proposition~5.8.9]{Sch93b} that
contains $x_0$ and $\mathsf{z}$, it also contains 
$x_0-\mathsf{z}$. Hence
$\EE\alpha_0=\EE\|x_0-\mathsf{z}\|^p<\pinf$,
and it follows from Lemma~\ref{l:2016}\ref{le:1ii} that
$\sum_{n\in\NN}\EE\|x_n-\mathsf{z}\|^p<\pinf$ and
$\sum_{n\in\NN}\EE\vartheta_n<\pinf$. Consequently,
\begin{equation}
\label{e:2016-05-01}
\EE\|x_n-\mathsf{z}\|^p\to 0.
\end{equation}

\ref{l:1iib}: In view of \eqref{e:sqf1},
since $(\eta_n)_{n\in\NN}\in\ell_+^1({\mathscr{F}})$, it follows
from \cite[Proposition~3.1(iii)]{Pafa16} that 
$(\|x_n-\mathsf{z}\|)_{n\in\NN}$ converges \as\,  
However, we derive from \eqref{e:2016-05-01} that there exists a 
strictly increasing sequence $(k_n)_{n\in\NN}$ in $\NN$ such that
$\|x_{k_n}-\mathsf{z}\|\to 0$ $\as$ 
\cite[Corollaire~5.8.11]{Sch93b}. 
Altogether $\|x_{n}-\mathsf{z}\|\to 0$ $\as$ 
\end{proof}

\begin{theorem} 
\label{t:1}
Let  $(\lambda_n)_{n\in\NN}$ be a sequence in $\rzeroun$ such that 
$\inf_{n\in\NN}\lambda_n>0$, and let $(t_{n})_{n\in\NN}$, 
$(x_{n})_{n\in\NN}$, and $(e_{n})_{n\in\NN}$ be sequences of 
$\HH$-valued random variables. Further, let 
$(\FF_n)_{n\in\NN}$ be a sequence of sub-sigma-algebras of 
$\FF$ such that
\begin{equation}
(\forall n\in\NN)\quad\sigma(x_0,\ldots,x_n)\subset
\FF_n\subset\FF_{n+1}.
\end{equation}
Suppose that the following are satisfied:
\begin{enumerate}[label=\rm{[\alph*]}]
\item
\label{t:1a}
$(\forall n\in\NN)$ $x_{n+1}=x_n+\lambda_n(t_n+e_n-x_n)$.
\item
\label{t:1b}
There exists a sequence $(\xi_n)_{n\in\NN}$ in $\RP$ such that
\begin{equation}
\label{e:sumlaepsilon}
\sum_{n\in\NN}\sqrt{\xi_n}<\pinf
\end{equation}
and $(\forall n\in\NN)$ $\EC{\|e_n\|^2}{\FF_n}\leq\xi_n$.
\item
\label{t:1c}
There exist $\mathsf{z}\in\mathsf{H}$,
$(\theta_n)_{n\in\NN}\in\ell_+({\mathscr{F}})$, 
$(\nu_n)_{n\in\NN}\in\ell_+({\mathscr{F}})$, and
a sequence $(\mu_n)_{n\in\NN}$ in $\RP$ such that 
$\varlimsup\mu_n<1$ and 
\begin{equation}
\label{e:tnzmunu}
(\forall n\in\NN)\quad
\EC{\|t_n-\mathsf{z}\|^2}{\FF_n}+\theta_n\leq\mu_n
\|x_n-\mathsf{z}\|^2+\nu_n\quad\as
\end{equation}
\end{enumerate}
Set
\begin{equation}
(\forall n\in\NN)\quad 
\begin{cases}
\chi_n=1-\lambda_n+\lambda_n\mu_n+\sqrt{\xi_n}\lambda_n\big(1-\lambda_n+\lambda_n
\sqrt{\mu_n}\big)\\
\displaystyle\overline{\vartheta}_n=\sum_{k=0}^n
\Bigg[\prod_{\ell=k+1}^n\chi_\ell\Bigg]
\lambda_k\big(\EC{\theta_k}{\FF_0}+(1-\lambda_k)
\EC{\|t_k-x_k\|^2}{\FF_0}\big)\\[5mm]
\displaystyle\overline{\eta}_n=
\sum_{k=0}^n\Bigg[\prod_{\ell=k+1}^n\chi_\ell\Bigg] 
\lambda_k\Big(\EC{\nu_k}{\FF_0}+\big(1-\lambda_k+\lambda_k\big(
2\EC{\sqrt{\nu_k}}{\FF_0}+\sqrt{\mu_k}\big)\big)\sqrt{\xi_k}
+\lambda_k\xi_k\Big).
\end{cases}
\end{equation}
Then the following hold:
\begin{enumerate}
\item 
\label{t:1i} 
$(\forall n\in\NN)$
$\displaystyle\EC{\|x_{n+1}-\mathsf{z}\|^2}{\FF_0}+
\overline{\vartheta}_n\leq\Bigg(\prod_{k=0}^n\chi_k\Bigg)
\|x_0-\mathsf{z}\|^2+\overline{\eta}_n\quad\as$
\item 
\label{t:1ii}  
Suppose that $x_0\in L^2(\Omega,\FF,\PP;\HH)$ and that
\begin{equation}
\label{e:sumsqrtllanu}
\sum_{n\in\NN}\sqrt{\EE\nu_n}<\pinf. 
\end{equation}
Then the following hold:
\begin{enumerate}
\item 
\label{t:1iia} 
$\EE\|x_n-\mathsf{z}\|^2\to 0$.
\item 
\label{t:1iia+} 
$\sum_{n\in\NN}\EE\theta_n<\pinf$. 
\item
\label{t:1iia++} 
$\sum_{n\in\NN}(1-\lambda_n)\EE\|t_n-x_n\|^2<\pinf$.
\item 
\label{t:1iib} 
Suppose that $(\nu_n)_{n\in\NN}\in\ell_+^{1/2}({\mathscr{F}})$. Then
$(x_n)_{n\in\NN}$ converges strongly $\as$ to $\mathsf{z}$.
\end{enumerate}
\end{enumerate}
\end{theorem}
\begin{proof}
\ref{t:1i}: 
Set $\lambda=\inf_{n\in\NN}\lambda_n$. Then
\begin{equation}
(\forall n\in\NN)\quad\chi_n\leq 1-(1-\mu_n){\lambda}
+\sqrt{\xi_n}\big(1+\sqrt{\mu_n}\big).
\end{equation}
Since $\varlimsup\mu_n<1$ and $\lim\xi_n=0$, we
have $\varlimsup\chi_n<1$. In addition, we derive 
from~\ref{t:1a}, \cite[Corollary~2.15]{Livre1}, and 
\eqref{e:tnzmunu} that
\begin{align}
\label{e:dusoir3007O}
(\forall n\in\NN)\quad&\|x_{n+1}-\mathsf{z}\|^2\nonumber\\
&=\|(1-\lambda_n) (x_n-\mathsf{z})+\lambda_n
(t_n-\mathsf{z})\|^2\nonumber\\
&\quad+2\lambda_n\scal{(1-\lambda_n)
(x_n-\mathsf{z})+\lambda_n\big(t_n-\mathsf{z}\big)}{e_n} 
+\lambda_n^2\|e_n\|^2\nonumber\\
&=(1-\lambda_n)\|x_n-\mathsf{z}\|^2
+\lambda_n\|t_n-\mathsf{z}\|^2-\lambda_n(1-\lambda_n)
\|t_n-x_n\|^2\nonumber\\
&\quad\;+2\lambda_n\scal{(1-\lambda_n)
(x_n-\mathsf{z})+\lambda_n\big(t_n-\mathsf{z}\big)}{e_n} 
+\lambda_n^2\|e_n\|^2\quad\as
\end{align}
Hence, \ref{t:1c} implies that
\begin{align}
\label{e:2013-11-16y}
&\hskip -4mm
(\forall n\in\NN)\quad
\EC{\|x_{n+1}-\mathsf{z}\|^2}{\FF_n}\nonumber\\
&\leq(1-\lambda_n)\|x_n-\mathsf{z}\|^2+
\lambda_n\EC{\|t_n-\mathsf{z}\|^2}{\FF_n}
-\lambda_n(1-\lambda_n)\EC{\|t_n-x_n\|^2}{\FF_n}\nonumber\\
&\quad\;+2\lambda_n\big((1-\lambda_n)\|x_n-\mathsf{z}\| 
+\lambda_n
\sqrt{\EC{\|t_n-\mathsf{z}\|^2}{\FF_n}}\big)\sqrt{\EC{\|e_n\|^2}{\FF_n}}
+\lambda_n^2\EC{\|e_n\|^2}{\FF_n}\nonumber\\
&\leq(1-\lambda_n)\|x_n-\mathsf{z}\|^2+\lambda_n\big(
\mu_n\|x_n-\mathsf{z}\|^2+\nu_n
-\theta_n\big)-\lambda_n(1-\lambda_n)\EC{\|t_n-x_n\|^2}{\FF_n}
\nonumber\\
&\quad\;+2\lambda_n\big((1-\lambda_n)\|x_n-\mathsf{z}\| 
+\lambda_n\sqrt{\mu_n
\|x_n-\mathsf{z}\|^2+\nu_n}\big)\sqrt{\EC{\|e_n\|^2}{\FF_n}}
+\lambda_n^2\EC{\|e_n\|^2}{\FF_n}\quad\as
\end{align}
Now set 
\begin{equation}
\label{e:2016-15-7a}
(\forall n\in\NN)\quad 
\begin{cases}
\vartheta_n=\lambda_n\theta_n+
\lambda_n(1-\lambda_n)\EC{\|t_n-x_n\|^2}{\FF_n}\\
\kappa_n=\lambda_n\nu_n+2\lambda_n^2\sqrt{\nu_n}
\sqrt{\EC{\|e_n\|^2}{\FF_n}}+\lambda_n^2\EC{\|e_n\|^2}{\FF_n}\\
\eta_n=\lambda_n\nu_n+\lambda_n
\big(1-\lambda_n+\lambda_n(2\sqrt{\nu_n}+\sqrt{\mu_n})\big)
\sqrt{\xi_n}+\lambda_n^2\xi_n.
\end{cases}
\end{equation}
It follows from \ref{t:1b} that
\begin{align}
\label{e:2013-11-16z}
(\forall n\in\NN)\quad
&\EC{\|x_{n+1}-\mathsf{z}\|^2}{\FF_n}\nonumber\\
&\leq(1-\lambda_n+\lambda_n\mu_n)\|x_n-\mathsf{z}\|^2
+2\lambda_n (1-\lambda_n+\lambda_n\sqrt{\mu_n})\|x_n-\mathsf{z}\| 
\sqrt{\EC{\|e_n\|^2}{\FF_n}}\nonumber\\
&\quad\;-\vartheta_n+\kappa_n\nonumber\\
&\leq(1-\lambda_n+\lambda_n\mu_n)\|x_n-\mathsf{z}\|^2
+\lambda_n (1-\lambda_n+\lambda_n\sqrt{\mu_n})(\|x_n-\mathsf{z}\|^2+1) 
\sqrt{\EC{\|e_n\|^2}{\FF_n}}\nonumber\\
&\quad\;-\vartheta_n+\kappa_n\nonumber\\
&\leq\chi_n\|x_n-\mathsf{z}\|^2 -\vartheta_n+\eta_n\quad\as
\end{align}
The result then follows by applying Lemma~\ref{l:1}\ref{l:1i} with
$\phi=|\cdot|^2$.

\ref{t:1iia}: 
According to \eqref{e:2016-15-7a}, for every $n\in\NN$,
\begin{align}
\label{e:hackett}
\EE\eta_n=&\;\lambda_n\EE\nu_n+\lambda_n
\big(1-\lambda_n+\lambda_n(2\EE\sqrt{\nu_n}+\sqrt{\mu_n})\big)
\sqrt{\xi_n}+\lambda_n^2\xi_n\nonumber\\
=&\;\lambda_n\EE\nu_n+(1-\lambda_n+\sqrt{\mu_n})\lambda_n\sqrt{\xi_n}
+2\lambda_n^2\sqrt{\xi_n}\,\EE\sqrt{\nu_n}
+\big(\lambda_n\sqrt{\xi_n}\big)^2\nonumber\\
\leq &\;\EE\nu_n+(1+\sqrt{\mu_n})\sqrt{\xi_n}
+2\bigg(\sup_{k\in\NN}\sqrt{\xi_k}\bigg)\sqrt{\EE\nu_n}
+\big(\sqrt{\xi_n}\big)^2,
\end{align}
where we have used the fact that $\lambda_n\in\rzeroun$ and Jensen's
inequality. We deduce from \eqref{e:hackett},
\eqref{e:sumlaepsilon}, and
\eqref{e:sumsqrtllanu} that $\sum_{n\in\NN}\EE\eta_n<\pinf$. 
Hence it follows from \eqref{e:2013-11-16z} and
Lemma~\ref{l:1}\ref{l:1iia} that
$\EE\|x_n-\mathsf{z}\|^2\to 0$ and that
$\sum_{n\in\NN}\EE\vartheta_n<\pinf$. 
In view of \eqref{e:2016-15-7a}, we obtain
\ref{t:1iia+} and \ref{t:1iia++}.

\ref{t:1iib}: In view of \eqref{e:2016-15-7a}, if 
$(\nu_n)_{n\in\NN}\in\ell_+^{1/2}({\mathscr{F}})$, then
$(\eta_n)_{n\in\NN}\in\ell_+^1({\mathscr{F}})$ and the strong 
convergence claim follows from Lemma~\ref{l:1}\ref{l:1iib}.
\end{proof}

\begin{remark}\ 
\begin{enumerate}
\item 
Under the assumptions of Theorem~\ref{t:1}, 
if $\nu_n\equiv 0$ and $\xi_n\equiv 0$, then
$\overline{\eta}_n\equiv 0$ and it follows from 
\ref{t:1i} that
$(\EC{\|x_{n}-\mathsf{z}\|^2}{\FF_0})_{n\in\NN}$ 
converges linearly to 0.
\item
The weak and strong almost sure convergences of a 
sequence $(x_n)_{n\in\NN}$ governed by \ref{t:1a} 
and \eqref{e:tnzmunu} were established in
\cite[Theorem~2.5]{Siop15} under different assumptions on
$(\mu_n)_{n\in\NN}$, $(\nu_n)_{n\in\NN}$, and $(e_n)_{n\in\NN}$.
\end{enumerate}
\end{remark}

\section{Mean-square and linear convergence of
Algorithm~\ref{algo:1}}
\label{sec:3}

We complement the almost sure weak and strong convergence results 
of \cite{Siop15} on the convergence of the orbits of
Algorithm~\ref{algo:1} by establishing mean-square and linear
convergence properties.
 
\subsection{Main results}
\label{se:mainres}

The next theorem constitutes our main result in terms of 
mean-square convergence. For added flexibility, this convergence 
will be evaluated in a norm $|||\cdot|||$ on $\HHH$ parameterized
by weights $(\omega_i)_{1\leq i\leq m}\in\RPP^m$ and 
defined by
\begin{equation}
\label{e:2014-02-14}
(\forall\boldsymbol{\mathsf{x}}\in\HHH)\quad
|||\boldsymbol{\mathsf{x}}|||^2=
\sum_{i=1}^m\omega_i\|\mathsf{x}_i\|^2.
\end{equation}

\begin{theorem}
\label{t:3}
Consider the setting of Problem~\ref{prob:1} and 
Algorithm~\ref{algo:1}, and let $(\FF_n)_{n\in\NN}$ be a sequence
of sub-sigma-algebras of $\FF$ such that
\begin{equation}
(\forall n\in\NN)\quad
\sigma(\boldsymbol{x}_0,\ldots,\boldsymbol{x}_n)\subset
\FF_n\subset\FF_{n+1}.
\end{equation}
Assume that the following are satisfied:
\begin{enumerate}[label=\rm{[\alph*]}]
\item 
\label{t:3a} 
$\inf_{n\in\NN}\lambda_n>0$.
\item
\label{t:3b}
There exists a sequence $(\alpha_n)_{n\in\NN}$ in $\RP$ such that 
$\sum_{n\in\NN}\sqrt{\alpha_n}<\pinf$ and, for every $n\in\NN$,
$\EC{\|\boldsymbol{a}_n\|^2}{\FF_n}\leq\alpha_n$.
\item
\label{t:3iv}
For every $n\in\NN$,
$\mathcal{E}_n=\sigma(\boldsymbol{\varepsilon}_n)$ and $\FF_n$
are independent.
\item
\label{t:3v}
For every $i\in\{1,\ldots,m\}$,
$\mathsf{p}_i=\PP[\varepsilon_{i,0}=1]>0$.
\end{enumerate}
Then the following hold:
\begin{enumerate}
\item 
\label{t:3i}  
Let $(\omega_i)_{1\leq i\leq m}\in\RPP^m$ be such that
\begin{equation}
\label{e:2017-3-3vi}
\begin{cases}
(\forall i\in\{1,\ldots,m\})\quad
\varlimsup\tau_{i,n}<\omega_i\mathsf{p}_i\\
\displaystyle\max_{1\leq i\leq m}\omega_i\mathsf{p}_i=1,
\end{cases}
\end{equation}
set
\begin{equation}
\label{e:defmuntaui}
(\forall n\in\NN)\quad 
\begin{cases}
\xi_n=\alpha_n\displaystyle{\max_{1\leq i\leq m}}\;\omega_i\\
\displaystyle\mu_n= 1-\min_{1\leq i\leq m}
\Big(\mathsf{p}_i-\frac{\tau_{i,n}}{\omega_i}\Big),
\end{cases}
\end{equation}
and define
\begin{equation}
\label{e:2016-7-16a}
(\forall n\in\NN)\quad 
\begin{cases}
\chi_n=1-\lambda_n(1-\mu_n)+\sqrt{\xi_n}\lambda_n 
(1-\lambda_n+\lambda_n\sqrt{\mu_n})\\
\displaystyle\overline{\eta}_n=\sum_{k=0}^n
\Bigg[\prod_{\ell=k+1}^n\chi_\ell\Bigg] 
\lambda_k\big(1-\lambda_k+\lambda_k\sqrt{\mu_k}
+\lambda_k\sqrt{\xi_k}\big)\sqrt{\xi_k}.
\end{cases}
\end{equation}
Then
\begin{equation}
\label{e:decnormw}
(\forall n\in\NN)\;\;
\displaystyle\sum_{i=1}^m\omega_i
\EC{\|x_{i,n+1}-\overline{\mathsf{x}}_i\|^2}{\FF_0}
\leq\Bigg(\prod_{k=0}^n\chi_k\Bigg)\Bigg(\sum_{i=1}^m\omega_i
\|x_{i,0}-\overline{\mathsf{x}}_{i,0}\|^2\Bigg)+
\overline{\eta}_n\quad\as
\end{equation}
\item 
\label{t:3ii}  
Suppose that $\boldsymbol{x}_0\in L^2(\Omega,\FF,\PP;\HHH)$
and $(\forall i\in\{1,\ldots,m\})$
$\varlimsup\tau_{i,n}<1$.
Then
$\EE\|\boldsymbol{x}_n-\overline{\boldsymbol{\mathsf{x}}}\|^2\to
0$ and 
$\boldsymbol{x}_n\to\overline{\boldsymbol{\mathsf{x}}}$ $\as$
\end{enumerate}
\end{theorem}
\begin{proof}
\ref{t:3i}:
We are going to apply Theorem~\ref{t:1} in the Hilbert
space $(\HHH,|||\cdot|||)$ defined by \eqref{e:2014-02-14}. Set
\begin{equation}
\label{e:2014-02-16}
(\forall n\in\NN)\quad
\boldsymbol{t}_n=\big(x_{i,n}+\varepsilon_{i,n}({\mathsf T}_{\!i,n}\,
\boldsymbol{x}_n-x_{i,n})\big)_{1\leq i\leq m}
\quad\text{and}\quad
\boldsymbol{e}_n=(\varepsilon_{i,n}a_{i,n})_{1\leq i\leq m}.
\end{equation}
Then it follows from \eqref{e:2014} that
\begin{equation}
\label{e:2014-02-17a}
(\forall n\in\NN)\quad
\boldsymbol{x}_{n+1}=\boldsymbol{x}_{n}+\lambda_n\big(
\boldsymbol{t}_n+\boldsymbol{e}_n-\boldsymbol{x}_{n}\big),
\end{equation}
while \ref{t:3b} implies that
\begin{equation}
\label{e:2014-02-17b}
(\forall n\in\NN)\qquad
\EC{|||\boldsymbol{e}_n|||^2}{\FF_n}\leq
\EC{|||\boldsymbol{a}_n|||^2}{\FF_n}\leq 
\alpha_n\max_{1\leq i\leq m}\omega_i=\xi_n.
\end{equation}
We note that it also follows from \ref{t:3b} that $\sum_{n\in\NN}
\sqrt{\xi_n}<\pinf$. Now define
\begin{equation}
\label{e:2014-04-13a}
(\forall n\in\NN)(\forall i\in\{1,\ldots,m\})\quad
\mathsf{q}_{i,n}\colon\HHH\times\mathsf{D}\to\RR\colon
(\boldsymbol{\mathsf{x}},\boldsymbol{\epsilon})\mapsto
\|\mathsf{x}_i-\overline{\mathsf{x}}_i+
\epsilon_i({\mathsf T}_{\!i,n}\,
\boldsymbol{\mathsf{x}}-\mathsf{x}_i)\|^2.
\end{equation}
Then, for every $n\in\NN$ and every $i\in\{1,\ldots,m\}$, the
measurability of ${\mathsf T}_{\!i,n}$ implies that of the
functions 
$(\mathsf{q}_{i,n}(\cdot,\boldsymbol{\epsilon}))%
_{\boldsymbol{\epsilon}\in\mathsf{D}}$. However,
for every $n\in\NN$, \ref{t:3iv} asserts that the events 
$([\boldsymbol{\varepsilon}_n=
\boldsymbol{\epsilon}])_{\boldsymbol{\epsilon}\in\mathsf{D}}$
constitute an almost sure partition of $\Omega$ and are independent 
from $\FF_n$, while the random variables 
$(\mathsf{q}_{i,n}(\boldsymbol{x}_n,\boldsymbol{\epsilon}))%
_{\substack{1\leq i\leq m\\\boldsymbol{\epsilon}\in\mathsf{D}}}$ 
are $\FF_n$-measurable. Therefore, we derive from 
\cite[Section~28.2]{LoevII} that
\begin{align}
\label{e:2014-04-15c}
&\hskip -6mm
(\forall n\in\NN)(\forall i\in\{1,\ldots,m\})\quad
\EC{\|x_{i,n}+\varepsilon_{i,n}({\mathsf T}_{\!i,n}\,
\boldsymbol{x}_n-x_{i,n})-\overline{\mathsf{x}}_{i}\|^2}{\FF_n}
\nonumber\\
&\hskip 53mm=\EEC{\mathsf{q}_{i,n}(\boldsymbol{x}_n,\boldsymbol
{\varepsilon}_n)\sum_{\boldsymbol{\epsilon}\in\mathsf{D}}
1_{[\boldsymbol{\varepsilon}_n=\boldsymbol{\epsilon}]}}{\FF_n}
\nonumber\\
&\hskip 53mm=\sum_{\boldsymbol{\epsilon}\in\mathsf{D}}
\EC{\mathsf{q}_{i,n}(\boldsymbol{x}_n,\boldsymbol{\epsilon})
1_{[\boldsymbol{\varepsilon}_n=\boldsymbol{\epsilon}]}}{\FF_n}
\nonumber\\
&\hskip 53mm=\sum_{\boldsymbol{\epsilon}\in\mathsf{D}}
\EC{1_{[\boldsymbol{\varepsilon}_n=\boldsymbol{\epsilon}]}}{\FF_n}
\mathsf{q}_{i,n}(\boldsymbol{x}_n,\boldsymbol{\epsilon})\nonumber\\
&\hskip 53mm=\sum_{\boldsymbol{\epsilon}\in\mathsf{D}}
\PP[\boldsymbol{\varepsilon}_n=\boldsymbol{\epsilon}]
\mathsf{q}_{i,n}(\boldsymbol{x}_n,\boldsymbol{\epsilon})\quad\as
\end{align}
Combining this identity with
\eqref{e:2014-02-14}, \eqref{e:2014-02-16}, \ref{t:3v},
\eqref{e:2017-3-3vi}, and \eqref{e:1} yields
\begin{align}
\label{e:2014-02-18}
(\forall n\in\NN)\quad
&\EC{|||\boldsymbol{t}_n-\overline{\boldsymbol{\mathsf{x}}}|||^2}
{\FF_n}\nonumber\\
&=\sum_{i=1}^m\omega_i
\EE\Big(\big\|x_{i,n}+\varepsilon_{i,n}({\mathsf T}_{\!i,n}\,
\boldsymbol{x}_n-x_{i,n})-\overline{\mathsf{x}}_i\big\|^2
~\Big|~\FF_n\Big)\nonumber\\
&=\sum_{i=1}^m\omega_i
\sum_{\boldsymbol{\epsilon}\in\mathsf{D}}
\PP[\boldsymbol{\varepsilon}_n=\boldsymbol{\epsilon}]
\mathsf{q}_{i,n}(\boldsymbol{x}_n,\boldsymbol{\epsilon})
\nonumber\\
&=\sum_{i=1}^m\omega_i
\left(\sum_{\boldsymbol{\epsilon}\in\mathsf{D},\epsilon_i=1}
\PP[\boldsymbol{\varepsilon}_n=\boldsymbol{\epsilon}]\,
\|\mathsf{T}_{\!i,n}\,\boldsymbol{x}_n-\overline{\mathsf{x}}_i\|^2+
\sum_{\substack{\boldsymbol{\epsilon}\in\mathsf{D},\,\epsilon_i=0}}
\PP[\boldsymbol{\varepsilon}_n=\boldsymbol{\epsilon}]\,
\|x_{i,n}-\overline{\mathsf{x}}_i\|^2\right)\nonumber\\
&=\sum_{i=1}^m\omega_i\mathsf{p}_i
\|\mathsf{T}_{\!i,n}\,\boldsymbol{x}_n-\overline{\mathsf{x}}_i\|^2+
\sum_{i=1}^m\omega_i(1-\mathsf{p}_i)
\|x_{i,n}-\overline{\mathsf{x}}_i\|^2\nonumber\\
&\leq\bigg(\max_{1\leq i\leq m}\omega_i\mathsf{p}_i\bigg)
\sum_{i=1}^m\|\mathsf{T}_{\!i,n}\,\boldsymbol{x}_n-
\overline{\mathsf{x}}_i\|^2+\sum_{i=1}^m\omega_i(1-\mathsf{p}_i)
\|x_{i,n}-\overline{\mathsf{x}}_i\|^2\nonumber\\
&=|||\boldsymbol{x}_n-\overline{\boldsymbol{\mathsf{x}}}|||^2+
\|\boldsymbol{\mathsf T}_{\!n}\boldsymbol{x}_n-
\overline{\boldsymbol{\mathsf{x}}}\|^2
-\sum_{i=1}^m\omega_i\mathsf{p}_i
\|x_{i,n}-\overline{\mathsf{x}}_i\|^2\nonumber\\
&\leq|||\boldsymbol{x}_n-\overline{\boldsymbol{\mathsf{x}}}|||^2+
\sum_{i=1}^m (\tau_{i,n}-\omega_i\mathsf{p}_i)
\|x_{i,n}-\overline{\mathsf{x}}_{i}\|^2\nonumber\\
&=\sum_{i=1}^m\omega_i\Big(1+\frac{\tau_{i,n}}
{\omega_i}-\mathsf{p}_i\Big)\|x_{i,n}-
\overline{\mathsf{x}}_{i}\|^2\nonumber\\
&\leq\left(1-\min_{1\leq i\leq m}\Big(\mathsf{p}_i
-\frac{\tau_{i,n}}{\omega_i}\Big)\right)\,
|||\boldsymbol{x}_n-\overline{\boldsymbol{\mathsf{x}}}|||^2\quad\as
\end{align}
Altogether, properties \ref{t:1a}--\ref{t:1c}
of Theorem~\ref{t:1} are satisfied with 
\begin{equation}
\label{e:defxin}
(\forall n\in\NN)\quad\theta_n=\nu_n=0.
\end{equation}
On the other hand, it follows from \eqref{e:2017-3-3vi} and
\eqref{e:defmuntaui} that $\varlimsup\mu_n<1$. Hence, we derive
from Theorem~\ref{t:1}\ref{t:1i} that
\begin{equation}
(\forall n\in\NN)\quad
\EC{|||\boldsymbol{x}_{n+1}-\overline{\boldsymbol{\mathsf{x}}}|||^2}
{\FF_0}\leq\Bigg(\prod_{k=0}^n\chi_k\Bigg)
|||\boldsymbol{x}_0-\overline{\boldsymbol{\mathsf{x}}}|||^2+
\overline{\eta}_n\quad\as
\end{equation}

\ref{t:3ii}: Consider \ref{t:3i} when 
$(\forall i\in\{1,\ldots,m\})$ $\omega_i=1/\mathsf{p}_i$.
The convergence then follows from the inequalities 
\begin{equation}
\label{e:normeq}
(\forall\boldsymbol{\mathsf{x}}\in\HHH)\quad
\min_{1\leq i\leq m}\mathsf{p}_i\, |||\boldsymbol{\mathsf{x}}
|||\leq\|\boldsymbol{\mathsf{x}}\|\leq\max_{1\leq i\leq m}
\mathsf{p}_i\, |||\boldsymbol{\mathsf{x}}|||
\end{equation}
and Theorem~\ref{t:1}\ref{t:1ii}.
\end{proof}

\subsection{Linear convergence}
As an offspring of the results in Section~\ref{se:mainres}, we
obtain the following perturbed linear convergence result.

\begin{corollary}
\label{c:eqnormconvTn}
Consider the setting of Problem~\ref{prob:1} and 
Algorithm~\ref{algo:1}, suppose that \ref{t:3a}-\ref{t:3v} in
Theorem~\ref{t:3} are satisfied, and define $(\chi_n)_{n\in\NN}$ 
and $(\overline{\eta}_n)_{n\in\NN}$ as in 
\eqref{e:2016-7-16a}, where
\begin{equation}
\label{e:9fh}
\max_{1\leq i\leq m}\varlimsup\tau_{i,n}<1
\quad\text{and}\quad(\forall n\in\NN)\quad 
\begin{cases}
\xi_n=\dfrac{\alpha_n}{\displaystyle\min_{1\leq i\leq m}
\mathsf{p}_i}\\
\displaystyle\mu_n= 1-\min_{1\leq i\leq m}\mathsf{p}_i
\big(1-\tau_{i,n}\big).
\end{cases}
\end{equation}
Then
\begin{equation}
(\forall n\in\NN)\quad
\displaystyle\EC{\|\boldsymbol{x}_{n+1}-
\overline{\boldsymbol{\mathsf{x}}}\|^2}{\FF_0}
\leq\frac{\displaystyle\max_{1\leq i\leq m}\mathsf{p}_i}
{\displaystyle\min_{1\leq i\leq m} 
\mathsf{p}_i}\Bigg(\prod_{k=0}^n\chi_k\Bigg)
\|\boldsymbol{x}_0-\overline{\boldsymbol{\mathsf{x}}}\|^2+
\overline{\eta}_n\quad\as
\end{equation}
\end{corollary}
\begin{proof}
In view of \eqref{e:normeq}, the claim follows from 
Theorem~\ref{t:3}\ref{t:3i} applied with 
$(\forall i\in\{1,\ldots,m\})$ $\omega_i=1/\mathsf{p}_i$.
\end{proof}

Let us now make some observations to assess the consequences of
Corollary~\ref{c:eqnormconvTn} in terms of bounds on convergence 
rates, and the potential impact of the activation probabilities  
of the blocks $(\mathsf{p}_i)_{1\leq i\leq m}$ 
on them. Let us consider the case when $\alpha_n\equiv 0$, 
i.e., when there are no errors. Set
\begin{equation}
\label{e:bordel}
(\forall n\in\NN)\quad\chi_n=1-\lambda_n\min_{1\leq i\leq m}
\mathsf{p}_i(1-\tau_{i,n}).
\end{equation}
Then we derive from \eqref{e:2016-7-16a} and \eqref{e:9fh} that 
\begin{equation}
(\forall n\in\NN)\quad
\EC{\|\boldsymbol{x}_{n+1}-
\overline{\boldsymbol{\mathsf{x}}}\|^2}{\FF_0}
\leq\frac{\displaystyle\max_{1\leq i\leq m} 
\mathsf{p}_i}{\displaystyle\min_{1\leq i\leq m}\mathsf{p}_i} 
\Bigg(\prod_{k=0}^n\chi_k\Bigg)\|\boldsymbol{x}_0-
\overline{\boldsymbol{\mathsf{x}}}\|^2\quad\as
\end{equation}
Since \eqref{e:9fh} yields $\sup_{n\in\NN}\chi_n<1$, 
a linear convergence rate is thus obtained. 

For simplicity, let us further assume that the blocks are processed 
uniformly in the sense that $(\forall i\in\{1,\ldots,m\})$
$\mathsf{p}_i=\mathsf{p}$. Set
\begin{equation}
\chi=1-\inf_{n\in\NN}\bigg(\lambda_n
\Big(1-\max_{1\leq i\leq m}\tau_{i,n}\Big)\bigg)\in\lzeroun.
\end{equation}
Then
\begin{equation}
\label{e:tau}
(\forall n\in\NN)\quad\chi_n=1-\lambda_n\mathsf{p}\Big(
1-\max_{1\leq i\leq m}\tau_{i,n}\Big) \le 1-(1-\chi)\mathsf{p}.
\end{equation}
When $\mathsf{p}=1$, the upper bound in \eqref{e:tau} on the
convergence rate is minimal and equal to $\chi$. This is consistent
with the intuition that
frequently activating the coordinates should favor the convergence
speed as a function of the iteration number.  
On the other hand, activating the blocks less frequently induces a
reduction of the computational load per iteration. In
large scale problems, this reduction may actually be imposed 
by limited computing or memory resources. In
Algorithm~\ref{algo:1}, the cost of computing  
${\mathsf T}_{\!i,n}(x_{1,n},\ldots,x_{m,n})$ is on the average
$\mathsf{p}$ times smaller than in the standard non
block-coordinate approach. Hence, if we assume that this cost 
is independent of $i$ and the iteration
number $n$, $N$ iterations of the block-coordinate algorithm
have the same computational cost as $\mathsf{p}N$ iterations
of a non block-coordinate approach. In view of \eqref{e:tau}, let us
introduce the quantity
\begin{equation}
\varrho(\mathsf{p})=-\frac{\ln\big(1-(1-{\chi})\mathsf{p}\big)}
{\mathsf{p}}
\end{equation}
to evaluate the convergence rate normalized by the
probability $\mathsf{p}$ accounting for computational cost.
Under the above assumptions, \eqref{e:tau} yields
\begin{equation}
(\forall n\in\NN)\quad\prod_{k=0}^n\chi_k\leq
\exp\big(-\varrho(\mathsf{p})\mathsf{p}(n+1)\big).
\end{equation}
Elementary calculations show that, if ${\chi}\neq 0$,
\begin{equation}
\label{e:unifboundrate}
-\frac{1-{\chi}}{\ln{\chi}}\leq
\frac{\varrho(\mathsf{p})}{\varrho(1)}\leq 1.
\end{equation}
For example, if ${\chi}>0.2$, then
$\varrho(\mathsf{p})/\varrho(1)\in[0.49,1]$. 
This shows that, for values of ${\chi}$ not too small, the decrease
in the normalized convergence rate remains limited with respect to
a deterministic approach in which all the blocks are activated. This
fact is illustrated by Figure~\ref{fig:1}, where the graph of
$\varrho$ is plotted for several values of $\chi$.

\begin{figure}[htb]
\centering
\includegraphics[width=15cm]{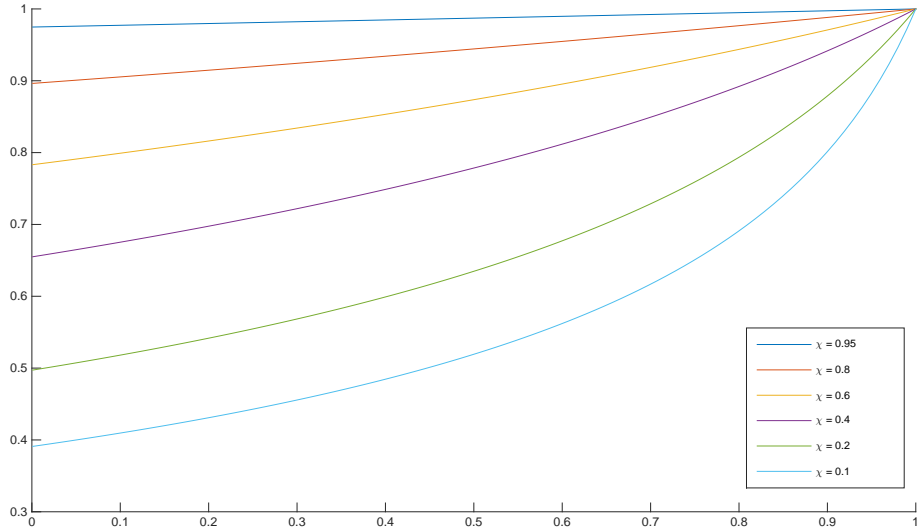}
\caption{Variations of $\varrho(\mathsf{p})/\varrho(\mathsf{1})$
as a function of $\mathsf{p}$ for various values of $\chi$.}
\label{fig:1}
\end{figure}

\begin{remark}
Let us consider the special case in which, for every 
$i\in \{1,\ldots,m\}$, $\tau_{i,n} \equiv
\tau_i$. Then \eqref{e:bordel} becomes
\begin{equation}
(\forall n\in\NN)\quad\chi_n=1-\lambda_n\min_{1\leq i\leq m}
\mathsf{p}_i(1-\tau_{i}).
\end{equation}
Now, let us further assume that, at each iteration $n$, only one of
the operators $(\mathsf{T}_{i,n})_{1\leq i\leq m}$ is activated
randomly. In this case, $\sum_{i=1}^m\mathsf{p}_i=1$ and
choosing
\begin{equation}
(\forall i \in \{1,\ldots,m\})\quad \mathsf{p}_i =
\frac{(1-\tau_i)^{-1}}{\sum_{j=1}^m (1-\tau_j)^{-1}}.
\end{equation}
leads to a minimum value of $\chi_n$.
\end{remark}

\section{Applications}
\label{sec:4}

In variational analysis, commonly encountered operators include
resolvent of monotone operators, projection operators, proximity 
operators of convex functions, gradient operators, and
various compositions and combinations thereof \cite{Livre1,Rock09}.
Specific instances of such operators used in iterative processes
which satisfy property \eqref{e:1} can be found in 
\cite{Baus16,Livre1,Baus12,Botr17,Botr15,Opti14,Pesq12,Rock76,%
Rock09,Sibo70}. 
In this section we highlight a couple of examples in the area of
splitting methods for systems of monotone inclusions. The
notation is that used in Problem~\ref{prob:1}. In addition, let
$\mathsf{A}\colon\HH\to 2^{\HH}$ be a set-valued operator.
We denote by 
$\zer\mathsf{A}=\menge{\mathsf{x}\in\HH}{0\in\mathsf{A}\mathsf{x}}$
the set of zeros of $\mathsf{A}$ and by 
$\mathsf{J}_{\mathsf{A}}=(\Id+\mathsf{A})^{-1}$ the resolvent of
$\mathsf{A}$. Recall that, if $\mathsf{A}$ is maximally
monotone, then $\mathsf{J}_{\mathsf{A}}$ is defined everywhere on 
$\HH$ and nonexpansive \cite{Livre1}. In the particular case when
$\mathsf{A}$ is the Moreau subdifferential $\partial \mathsf{f}$
of a proper lower semicontinuous convex function 
$\mathsf{f}\colon\HH\to\RX$, $\mathsf{J}_{\mathsf{A}}$
is the proximity operator $\prox_{\mathsf{f}}$ of $\mathsf{f}$ 
\cite{Livre1,Mor62b}.

\begin{example}
\label{ex:jbb}
For every $i\in\{1,\ldots,m\}$, let $\mathsf{A}_i\colon\HH\to
2^{\HH}$ be a maximally monotone operator, and consider the
coupled inclusion problem
\begin{equation}
\label{e:2017-04-18a}
\text{find}\;\boldsymbol{\mathsf{x}}
=(\mathsf{x}_i)_{1\leq i\leq m}\in\HHH\;\text{such that}\;
\begin{cases}
0\in\mathsf{A}_1\mathsf{x}_1+\mathsf{x}_1-\mathsf{x}_2\\
0\in\mathsf{A}_2\mathsf{x}_2+\mathsf{x}_2-\mathsf{x}_3\\
\hskip 4mm\vdots&\\
0\in\mathsf{A}_{m-1}\mathsf{x}_{m-1}+\mathsf{x}_{m-1}-\mathsf{x}_m\\
0\in\mathsf{A}_m\mathsf{x}_m+\mathsf{x}_m-\mathsf{x}_1.
\end{cases}
\end{equation}
For instance, in the case when each $\mathsf{A}_i$ is the normal 
cone operator to a nonempty closed convex set, 
\eqref{e:2017-04-18a} models limit
cycles in the method of periodic projections \cite{Bail12}. 
Another noteworthy instance is when
$m=2$, $\mathsf{A}_1=\partial\mathsf{f}_1$, and
$\mathsf{A}_2=\partial\mathsf{f}_2$, where
$\mathsf{f}_1$ and $\mathsf{f}_2$ are proper lower semicontinuous
functions from $\HH$ to $\RX$. Then \eqref{e:2017-04-18a} reduces 
to the joint minimization problem
\begin{equation}
\label{e:2017-04-18b}
\minimize{(\mathsf{x}_1,\mathsf{x}_2)\in\HH^2}
{\mathsf{f}_1(\mathsf{x}_1)+\mathsf{f}_2(\mathsf{x}_2)+
\frac{1}{2}\|\mathsf{x}_1-\mathsf{x}_2\|^2},
\end{equation}
studied in \cite{Acke80}. Now set
\begin{equation}
\label{e:2017-04-18c}
\boldsymbol{\mathsf{A}}\colon\boldsymbol{\mathsf{x}}
\mapsto(\mathsf{A}_1\mathsf{x}_1,\ldots,\mathsf{A}_m\mathsf{x}_m)
\quad\text{and}\quad
\boldsymbol{\mathsf{B}}\colon\boldsymbol{\mathsf{x}}
\mapsto(\mathsf{x}_1-\mathsf{x}_2,\mathsf{x}_2-\mathsf{x}_3,
\ldots,\mathsf{x}_m-\mathsf{x}_1).
\end{equation}
Then it follows from \cite[Proposition~20.23]{Livre1} that 
$\boldsymbol{\mathsf{A}}$ is maximally monotone. On the other
hand, $\boldsymbol{\mathsf{B}}$ is linear, bounded, and monotone
since 
\begin{equation}
\label{e:tonylevin}
(\forall\boldsymbol{\mathsf{x}}\in\HHH)\quad
\scal{\boldsymbol{\mathsf{B}}\boldsymbol{\mathsf{x}}}
{\boldsymbol{\mathsf{x}}}=\dfrac{\|\boldsymbol{\mathsf{B}}
\boldsymbol{\mathsf{x}}\|^2}{2}\geq 0. 
\end{equation}
It is therefore maximally
monotone \cite[Example~20.34]{Livre1}. Altogether, 
$\boldsymbol{\mathsf{A}}+\boldsymbol{\mathsf{B}}$ is maximally
monotone by \cite[Corollary~25.5(i)]{Livre1}. In addition,
suppose that each $\mathsf{A}_i$ is strongly monotone with
constant 
$\delta_i\in\RPP$. Then $\boldsymbol{\mathsf{A}}$ is strongly
monotone with constant $\min_{1\leq i\leq m}\delta_i$, and so
is $\boldsymbol{\mathsf{A}}+\boldsymbol{\mathsf{B}}$. We
therefore deduce from 
\cite[Corollary~23.37(ii)]{Livre1} that it possesses 
a unique zero $\overline{\boldsymbol{\mathsf{x}}}$, which is 
the unique solution to \eqref{e:2017-04-18a}. Let us
also note that, for every $i\in\{1,\ldots,m\}$, the resolvent 
$\mathsf{J}_{\mathsf{A}_i}$ is Lipschitz 
continuous with constant $\eta_{i}=1/(1+\delta_i)\in\zeroun$ 
\cite[Proposition~23.13]{Livre1}. Next, define
$\boldsymbol{\mathsf T}\colon\HHH\to\HHH\colon
\boldsymbol{\mathsf{x}}\mapsto({\mathsf T}_{\!i}\,
\boldsymbol{\mathsf{x}})_{1\leq i\leq m}$, where,
for every $i\in\{1,\ldots,m\}$, 
$\mathsf{T}_{\!i}\colon\HHH\to\HH_i\colon\boldsymbol{\mathsf{x}}
\mapsto\mathsf{J}_{\mathsf{A}_i}\mathsf{x}_{i+1}$, with
the convention $\mathsf{x}_{m+1}=\mathsf{x}_1$. Then
we derive from \eqref{e:2017-04-18a} that 
$\boldsymbol{\mathsf T}\overline{\boldsymbol{\mathsf{x}}}=
\overline{\boldsymbol{\mathsf{x}}}$. Moreover,
\begin{align}
\label{e:2}
(\forall n\in\NN)(\forall\boldsymbol{\mathsf{x}}\in\HHH)
\quad\|\boldsymbol{\mathsf{T}}\boldsymbol{\mathsf{x}}-
\overline{\boldsymbol{\mathsf{x}}}\|^2
&=\sum_{i=1}^m\|\mathsf{J}_{\mathsf{A}_i}\mathsf{x}_{i+1}-
\overline{\mathsf{x}}_i\|^2\nonumber\\
&=\sum_{i=1}^m\|\mathsf{J}_{\mathsf{A}_i}\mathsf{x}_{i+1}-
\mathsf{J}_{\mathsf{A}_i}\overline{\mathsf{x}}_{i+1}\|^2
\nonumber\\
&\leq\sum_{i=1}^m\eta^2_{i}
\|\mathsf{x}_{i+1}-\overline{\mathsf{x}}_{i+1}\|^2, 
\end{align}
which shows that \eqref{e:1} is satisfied upon choosing 
$\boldsymbol{\mathsf{T}}_n\equiv\boldsymbol{\mathsf{T}}$
and, for every $i\in \{1,\ldots,m\}$, $\tau_{i,n}\equiv\eta_i^2$.
In this scenario, Algorithm~\ref{algo:1} becomes
\begin{equation}
\label{e:2017-04-18x}
\begin{array}{l}
\text{for}\;n=0,1,\ldots\\
\left\lfloor
\begin{array}{l}
\text{for}\;i=1,\ldots,m-1\\
\left\lfloor
\begin{array}{l}
x_{i,n+1}=x_{i,n}+\varepsilon_{i,n}\lambda_n\big(
\mathsf{J}_{\mathsf{A}_i}x_{i+1,n}+a_{i,n}-x_{i,n}\big)
\end{array} 
\right.\\[1mm]
x_{m,n+1}=x_{m,n}+\varepsilon_{m,n}\lambda_n\big(
\mathsf{J}_{\mathsf{A}_m}x_{1,n}+a_{m,n}-x_{m,n}\big),
\end{array} 
\right.
\end{array} 
\end{equation}
and Theorem~\ref{t:3} describes its asymptotic behavior.
In the particular case of \eqref{e:2017-04-18b}, for $f_1$ and
$f_2$ strongly convex, \eqref{e:2017-04-18x} with 
$\lambda_n\equiv 1$ and no error, reduces to
\begin{equation}
\label{e:2017-04-18y}
\begin{array}{l}
\text{for}\;n=0,1,\ldots\\
\left\lfloor
\begin{array}{l}
x_{1,n+1}=x_{1,n}+\varepsilon_{1,n}\big(
\mathsf{prox}_{\mathsf{f}_1}x_{2,n}-x_{1,n}\big)\\[1mm]
x_{2,n+1}=x_{2,n}+\varepsilon_{2,n}\big(
\mathsf{prox}_{\mathsf{f}_2}x_{1,n}-x_{2,n}\big).
\end{array} 
\right.
\end{array} 
\end{equation}
In the deterministic setting in which $\varepsilon_{1,n}\equiv 1$
and $\varepsilon_{2,n}\equiv 1$, the resulting sequence
$(x_{2,n})_{n\in\NN}$ is that produced by the alternating
proximity operator method of \cite{Acke80}, further studied in 
\cite{Nona05}.
\end{example}

\begin{example}
\label{ex:7l}
We consider an $m$-agent model investigated in \cite{Sico10}. 
For every $i\in\{1,\ldots,m\}$, let 
$\mathsf{A}_i\colon\HH_i\to 2^{\HH_i}$ be a maximally monotone
operator modeling some abstract utility of agent $i$ and
let $\mathsf{B}_i\colon\HHH\to\HH_i$ be a coupling operator.
It is assumed that the operator
$\boldsymbol{\mathsf B}\colon\HHH\to\HHH\colon
\boldsymbol{\mathsf{x}}\mapsto({\mathsf B}_{i}\,
\boldsymbol{\mathsf{x}})_{1\leq i\leq m}$ is
$\beta$-cocoercive \cite{Livre1} for some $\beta\in\RPP$, that 
is,
\begin{equation}
(\forall\boldsymbol{\mathsf{x}}\in\HHH)
(\forall\boldsymbol{\mathsf{y}}\in\HHH)\quad
\scal{\boldsymbol{\mathsf{x}}-\boldsymbol{\mathsf{y}}}
{\boldsymbol{\mathsf B}\boldsymbol{\mathsf{x}}-
\boldsymbol{\mathsf B}\boldsymbol{\mathsf{y}}}\geq
\beta\|\boldsymbol{\mathsf B}\boldsymbol{\mathsf{x}}-
\boldsymbol{\mathsf B}\boldsymbol{\mathsf{y}}\|^2.
\end{equation}
The equilibrium problem is to 
\begin{equation}
\label{e:sico10}
\text{find}\;\boldsymbol{\mathsf{x}}\in\HHH
\;\;\text{such that}\;\;(\forall i\in\{1,\ldots,m\})\quad 0\in
\mathsf{A}_i\mathsf{x}_i+\mathsf{B}_i\big(\mathsf{x}_1,\ldots,
\mathsf{x}_m\big).
\end{equation}
For every $i\in\{1,\ldots,m\}$, let us further assume that
$\mathsf{A}_i$ is $\delta_i$-strongly monotone for some 
$\delta_i\in\RPP$ or, equivalently, that 
$\mathsf{M}_i=\mathsf{A}_i-\delta_i\Id$ is monotone. 
Since $\boldsymbol{\mathsf B}$ is maximally monotone
\cite[Example~20.31]{Livre1}, arguing as in 
Example~\ref{ex:jbb}, we arrive at the conclusion that 
$\boldsymbol{\mathsf A}+\boldsymbol{\mathsf B}$ has exactly one
zero $\overline{\boldsymbol{\mathsf{x}}}$, and that 
$\overline{\boldsymbol{\mathsf{x}}}$ is the unique solution 
to \eqref{e:sico10}. Let
\begin{equation}
\label{e:choicebetaFB}
\delta=\min_{1\leq i\leq m}\delta_i,
\quad\text{and}\quad(\forall n\in\NN)\quad
\theta_n\in [0,\delta]\quad\text{and}\quad
\gamma_n\in\RPP.
\end{equation}
Set
\begin{equation}
(\forall n\in\NN)\quad
\begin{cases}
\boldsymbol{\mathsf C}_n\colon\HHH\to 2^\HHH\colon
\boldsymbol{\mathsf x}\mapsto\underset{i=1}{\overset{m}{\Cart}}
\big(\mathsf{M}_i+(\delta_{i}-\theta_n)\Id\big)\mathsf{x}_i\\
\boldsymbol{\mathsf D}_n=\boldsymbol{\mathsf B}+\theta_n\ID\\[1mm]
\boldsymbol{\mathsf{T}}_n=
\boldsymbol{\mathsf{J}}_{\gamma_n\boldsymbol{\mathsf C}_n}
\circ(\ID-\gamma_n\boldsymbol{\mathsf D}_n).
\end{cases}
\end{equation}
Now let $n\in\NN$. We first observe that
\begin{equation}
\zer\big(\gamma_n\boldsymbol{\mathsf C}_n+
\gamma_n\boldsymbol{\mathsf D}_n\big)
=\zer\big(\boldsymbol{\mathsf A}+\boldsymbol{\mathsf B}\big)
=\{\overline{\boldsymbol{\mathsf{x}}}\}
=\Fix\boldsymbol{\mathsf{T}}_n,
\end{equation}
and derive from \cite[Proposition~23.17(i)]{Livre1} that 
\begin{equation}
\boldsymbol{\mathsf{J}}_{\gamma_n\boldsymbol{\mathsf C}_n}\colon
\boldsymbol{\mathsf{x}}\mapsto
\left(\mathsf{J}_{\frac{\gamma_n\mathsf{M}_i}
{1+\gamma_n(\delta_i-\theta_n)}}
\left(\frac{\mathsf{x}_i}{1+\gamma_n(\delta_i-\theta_n)}\right)
\right)_{1\leq i\leq m}.
\end{equation}
Hence \eqref{e:choicebetaFB} entails that 
$\boldsymbol{\mathsf{J}}_{\gamma_n\boldsymbol{\mathsf C}_n}$ is
Lipschitz continuous with constant 
$1/(1+\gamma_n({\delta}-\theta_n))$. 
On the other hand, since $\boldsymbol{\mathsf B}$ is 
$\beta$-cocoercive, there exists a nonexpansive operator
$\boldsymbol{\mathsf{R}}\colon\HHH\to\HHH$ such that 
$\beta\boldsymbol{\mathsf B}=(\ID+\boldsymbol{\mathsf{R}})/2$
\cite[Remark~4.34(iv)]{Livre1}. We have 
\begin{equation}
\ID-\gamma_n\boldsymbol{\mathsf D}_n=
\left(1-\gamma_n\theta_n-\frac{\gamma_n}{2\beta}\right)
\ID-\frac{\gamma_n}{2\beta}\boldsymbol{\mathsf R}.
\end{equation}
In turn, a Lipschitz constant of
$\ID-\gamma_n\boldsymbol{\mathsf D}_n$ is
$|1-\gamma_n(\theta_n+1/(2\beta))|+\gamma_n/(2\beta)$, and hence 
one for $\boldsymbol{\mathsf{T}}_{\!n}$ is
\begin{equation}
\zeta_n=\frac{\big|1-\gamma_n\big(\theta_n+1/(2\beta)\big)\big|+
\gamma_n/(2\beta)}{1+\gamma_n(\delta-\theta_n)}.
\end{equation}
Note that 
\begin{equation}
\zeta_n=
\begin{cases}
\dfrac{1-\gamma_n\theta_n}{1+\gamma_n(
\delta-\theta_n)}<1,&\text{if}\;\;
\gamma_n\leq\dfrac{2\beta}{1+2\beta\theta_n};\\[6mm]
\dfrac{\gamma_n(\theta_n+1/\beta)-1}
{1+\gamma_n(\delta-\theta_n)}<1,&\text{if}\;\;
\dfrac{2\beta}{1+2\beta\theta_n}<\gamma_n<
\dfrac{2\beta}{1+\beta(2\theta_n-\delta)}.
\end{cases}
\end{equation}
Consequently, imposing
\begin{equation}
\gamma_n<\dfrac{2\beta}{1+\beta(2\theta_n-\delta)}
\end{equation}
places us in the framework of Problem~\ref{prob:1} with
$(\forall i\in\{1,\ldots,m\})$ $\tau_{i,n}=\zeta^2_n$.
Algorithm~\ref{algo:1} for solving \eqref{e:sico10}, that is,
\begin{equation}
\label{e:main12}
\begin{array}{l}
\text{for}\;n=0,1,\ldots\\
\left\lfloor
\begin{array}{l}
\text{for}\;i=1,\ldots,m\\
\left\lfloor
\begin{array}{l}
\displaystyle x_{i,n+1}=x_{i,n}+\varepsilon_{i,n}\lambda_n
\bigg(\mathsf{J}_{\frac{\gamma_n\mathsf{M}_i}
{1+\gamma_n(\delta_i-\theta_n)}}\bigg(\frac{(1-\gamma_n
\theta_n)x_{i,n}-\gamma_n\mathsf{B}_i\boldsymbol{x}_n}
{1+\gamma_n(\delta_i-\theta_n)}\bigg)+a_{i,n}-x_{i,n}\bigg),
\end{array}
\right.
\end{array}
\right.\\
\end{array}
\end{equation}
is then an instance of the block-coordinate forward-backward
algorithm of \cite[Section~5.2]{Siop15}. Its convergence
properties in the present setting are given in Theorem~\ref{t:3}.
\end{example}

\begin{remark}
In view of \eqref{e:tonylevin}, \eqref{e:2017-04-18a} constitutes 
a special case of \eqref{e:sico10} and it can also be
solved via \eqref{e:main12}. In Example~\ref{ex:jbb}, we have
exploited the special structure of $\boldsymbol{\mathsf{B}}$ to
obtain tighter coefficients 
$(\tau_{i,n})_{1\leq i\leq m, n\in\NN}$ in \eqref{e:1}.
\end{remark}

\begin{example}
Let $\boldsymbol{\mathsf{g}}\colon\HHH\to\RR$ be a convex
function which is differentiable with a $\beta^{-1}$-Lipschitzian
gradient for some $\beta\in\RPP$ and, for every
$i\in\{1,\ldots,m\}$, let
$\mathsf{f}_i\colon\HH_i\to\RX$ be a proper lower semicontinuous
$\delta_i$-strongly convex function for some $\delta_i\in\RPP$. 
We consider the optimization problem
\begin{equation}
\label{e:pboptstcon}
\minimize{\mathsf{x}_1\in\HH_1,\ldots,\mathsf{x}_m\in\HH_m}
{\sum_{i=1}^m\mathsf{f}_i(\mathsf{x}_i)+\boldsymbol{\mathsf{g}}
(\mathsf{x}_1,\ldots,\mathsf{x}_m)}.
\end{equation}
Then it results from standard facts
\cite[Section~28.5]{Livre1} that this problem is the 
special case of Example~\ref{ex:7l} in which 
$\boldsymbol{\mathsf B}=\nabla\boldsymbol{\mathsf g}$
and, for every $i\in\{1,\ldots,m\}$,
$\mathsf{A}_i=\partial\mathsf{f}_i$.
Now set $(\forall i\in\{1,\ldots,m\})$
$\mathsf{h}_i=\mathsf{f}_i-\delta_i\|\cdot\|^2/2$. Then 
\eqref{e:main12} assumes the form 
\begin{equation}
\label{e:main122}
\begin{array}{l}
\text{for}\;n=0,1,\ldots\\
\left\lfloor
\begin{array}{l}
\text{for}\;i=1,\ldots,m\\
\left\lfloor
\begin{array}{l}
\displaystyle x_{i,n+1}=x_{i,n}+\varepsilon_{i,n}\lambda_n
\left(\prox_{\frac{\gamma_n\mathsf{h}_i}
{1+\gamma_n(\delta_i-\theta_n)}}
\bigg(\frac{(1-\gamma_n\theta_n)x_{i,n}-\gamma_n
\nabla_{\!i}\,\boldsymbol{\mathsf{g}}(\boldsymbol{x}_n)}
{1+\gamma_n(\delta_i-\theta_n)}\bigg)+a_{i,n}-x_{i,n}\right),
\end{array}
\right.
\end{array}
\right.\\
\end{array}
\end{equation}
where $\nabla_{\!i}\,\boldsymbol{\mathsf{g}}\colon\HHH\to\HH_i$ is
the $i$th component of $\nabla\boldsymbol{\mathsf{g}}$. 
\end{example}

\begin{remark}
In the case of a non block-coordinate implementation, i.e., $m=1$,
a mean-square convergence result for the forward-backward algorithm
can be found in \cite{Rosa16} under different assumptions than 
ours and, in particular, the requirement that the proximal 
parameters $(\gamma_n)_{n\in\NN}$ must go to $0$.
\end{remark}

\begin{remark}
In connection with the linear convergence of \eqref{e:main122} 
deriving from Corollary~\ref{c:eqnormconvTn}, let us note that a
similar result was obtained in \cite{Rich14} by imposing the
restrictions 
\begin{equation}
(\forall i\in\{1,\ldots,m\})\quad
\HH_i=\RR^{N_i},\quad\mathsf{p}_i=\frac{1}{m},
\quad\text{and}\quad(\forall n\in\NN)\quad
\lambda_n=1\quad\text{and}\quad a_{i,n}=0.
\end{equation}
In this specific setting the proximal parameter in \cite{Rich14} 
was chosen differently for each block: it is not allowed to vary 
with the iteration $n$ as in \eqref{e:main122}, but it can be
chosen differently for each $i$. 
In the case when $(\forall i \in \{1,\ldots,m\})$ $\mathsf{f}_i=0$, 
more freedom was given to the choice of $(\mathsf{p}_i)_{1\leq
i\leq m}$ in \cite{Rich14}, but by still activating only one block 
at each iteration.
Further narrowing the problem to the minimization of a 
smooth strongly convex function on $\RR^N$,
a coordinate descent method is proposed in \cite{Rich16}
which requires, for every $i\in \{1,\ldots,m\}$, $\HH_i=\RR$ and
allows for multiple coordinates to be randomly updated at 
each iteration, as in \eqref{e:main122}.
\end{remark}

\end{document}